\documentclass[12pt]{article}
\usepackage{graphicx}
\usepackage{amssymb,graphics}
\usepackage{amsmath}
\usepackage{graphicx}
\usepackage{amsthm}
\newcommand{\beqn}{\begin{eqnarray}}
\newcommand{\eeqn}{\end{eqnarray}}
\newcommand{\be}{\begin{equation}}
\newcommand{\ee}{\end{equation}}
\newcommand{\ba}{\begin{array}}
\newcommand{\ea}{\end{array}}
\newcommand{\pa}{\partial}
\newcommand{\bs}{\bigskip}

\newcommand{\bfr}{\begin{flushright}}
\newcommand{\efr}{\end{flushright}}
\newcommand{\bfl}{\begin{flushleft}}
\newcommand{\efl}{\end{flushleft}}

\newcommand{\R}{{\bf R}}

\newcommand{\T}{{\bf T}}

\newcommand{\D}{{\bf D}}
\newcommand{\F}{\mathbf{F}}
\newcommand{\G}{\mathbf{G}}

\newcommand{\intl}{\int\limits}

\newcommand{\curl}{{\rm curl\ }}

\def\phi{\varphi}
\newcommand{\eps}{\varepsilon}

\newcommand{\grad}{\mathrm{grad}}

\newcommand{\dive}{\mathrm{div}}

\newtheorem{theorem}{Theorem}[section]
\newtheorem{lemma}{Lemma}[section]

\newtheorem{corollary}{Corollary}[section]

\graphicspath{{converted_graphics/}}

\title{ON THE ANALYTICITY OF PARTICLE TRAJECTORIES\\ IN THE IDEAL INCOMPRESSIBLE FLUID}

\bs

\author{A. Shnirelman\\
Concordia University, Montreal, Canada\\
Institute for Advanced Study, Princeton, USA}

\begin{document}

\maketitle

\author

\section{Introduction}

The aim of this note is to give a new proof of the following striking fact. Consider the motion of ideal incompressible fluid in a bounded domain $M\subset\R^n$, or on a compact analytic Riemannian manifold (in this paper we restrict our analysis to the simplest case $M=\T^3$, for the ideas are the most transparent in this case, and the technical issues are reduced to the minimum). The flow is described by the Euler equations

\beqn
&&\frac{\pa u}{\pa t}+(u,\nabla)u+\nabla p=0,\\
&&\nabla \cdot u=0.
\eeqn
Suppose

\be
u(x,0)=u_0(x)
\ee
is the initial velocity. This problem has been studied since the seminal works of Gunter and Lichtenstein (\cite{L2}, \cite{Gun}). Here is a typical result. Suppose the initial velocity $u_0\in H^s, \ \ s>\frac{n}{2}+1$ (which is a little better than $C^1$). Then there exists a unique solution $u(x,t)\in H^s$ for $|t|<T(u_0)$. Moreover, solution $u(x,t)$ has exactly the same regularity as the initial velocity, i.e. if $u(x,t)\in H^r, \ \ r>s$ for some $t>0$, then $u_0\in H^r$. 

Now consider the trajectories of fluid particles defined by the equation

\be
\frac{dx}{dt}=u(x(t),t)
\ee
with the initial condition

\be
x(0)=x_0\in M.
\ee

\begin{theorem}\label{th1}
For any $x_0\in M$ the particle trajectory $x(t)$ is an analytic curve.
\end{theorem}

The first proof of this result was claimed by P.Serfati in 1993 (\cite{Ser1}, \cite{Ser2}). His proof is based of commutator estimates, and is consistently real; particle trajectories are real-analytic because they are "very smooth", i.e. their derivatives grow not very fast. Recently N.Nadirashvili proved that the flow lines of a {\it stationary} (time independent) solution are analytic [\cite{Nad}]. His proof is based on a detailed analysis of a semilinear elliptic equation satisfied by the stream function of the stationary solution, and on the classical idea (Levy-Petrovsky) of treating an elliptic equation as a hyperbolic one in an appropriate complex direction.

In this work the viewpoint is consistently complex: the particle trajectory $x(t)$ is analytic {\it because} it can be continued to complex values of $t$.  Our proof is based on the forgotten idea of Leon Lichtenstein which is buried in his paper \cite{L1} (published in 1925!). Lichtenstein himself has neither formulated nor proved this result; the adequate tools of complex analysis in the Banach spaces \cite{HP} were created much later.

\setcounter{equation}{0}

\section{Proof of analyticity of particle trajectories}

To show the proper meaning of the theorem, we have to use the Lagrangian description of the fluid flow. Consider the group $\D=SDiff^s(M)$ of volume preserving diffeomorphisms of $M$ of class $H^s$ (better to say, we always consider the component of unity of the group $\D$).  Let $Id$ be the unity in the group $\D$, i.e. the identity map. The Lie algebra of the group $\D$, $T_{Id}\D$ consists of vector fields $u(x)\in H^s$ on $M$ such that $\dive u=0$, and $\int_M udx=0$. We are always working in an $H^s$-neighborhood of $Id$ which in what follows goes without saying. The group $\D$ is equipped with the right-invariant $L^2$ metric (kinetic energy), defined for any $u, v \in T_{Id}\D$ as $(u, v)=\int_M u\cdot v dx$; the fluid flows are geodesics on $\D$ in this metric. They are described by the Lagrange equations (which express just the D'Alembert principle):

\be
\frac{\pa^2 g}{\pa t^2}+\nabla p \circ g=0,
\ee
where $g=g_t \in\D$ is the fluid trajectory, and $p=p(x,t)$ is a new unknown scalar function (pressure).

Consider the geodesic exponential map on $\D$. Let us denote by $g_t(v)$ the solution of the Lagrange equation (2.1) satisfying initial conditions

\be
g_0(v)=Id, \ \ \dot g_t(v)|_{t=0}=v.
\ee
Then the geodesic exponential map $Exp: T_{Id}\D \to \D$ is defined as

\be
Exp: v \mapsto g_1(v).
\ee
This map is defined for $||v||_s<\eps$ for some $\eps>0$.

\begin{theorem}
The map $Exp$ is real-analytic.
\end{theorem}

\begin{corollary}
For any $v\in T_{Id}\D$ the trajectory $g_t(v)$ is an analytic curve in $\D$.
\end{corollary}

\begin{proof}
Observe that $g_t(v)=Exp(t\cdot v)$; now we see that the trajectory $g_t(v)$ is an image of the segment $\{t \cdot v: \ 0\le t\le 1\}$ under the analytical map $Exp$.
\end{proof}

\begin{corollary}
If $u(x)$ is a stationary (time independent) solution of the Euler equations, then the flow lines are analytic curves in $M$.
\end{corollary}

\begin{proof}
For a stationary flow the flow lines are projections of trajectories $(t, x(t))$ from $\R\times M$ to $M$; hence, they are analytic beyond the critical points where $u=\frac{dx}{dt}=0$.
\end{proof}

\begin{proof}[Proof of Theorem 1.1]If the trajectory, i.e. the family of diffeomorphisms $g_t(v)=g\cdot t: | M \to M$ depends analytically on $t$, then the image of every point $x\in M$, $g_t(x)$, depends analytically on $t$, i.e. $g_t(x)$ is an analytical curve in $M$.
\end{proof}

In what follows, we consider the case $M=\T^3$, the 3-d torus; general case is done similarly, but here notations are simpler. On the torus, we restrict ourselves to the diffeomorphisms $g \in\D$ preserfing the center of mass: $\int_M(g(x)-x) dx=0$. For this group we keep the same notation $\D$. 

Following Lichtenstein \cite{L2}  and Gunter \cite{Gun}, we reduce our problem to the vorticity equation. Let $\omega=\curl u$. Then, by the Kelvin-Helmholtz theorem, $\omega(t)=g_{t*}\omega(0)$, i.e., in the case $M=\T^3$,

\be
\omega(g_t(x),t)=g_t'\cdot\omega(x,0),
\ee
where $g_t'$ is the Jacobi matrix of $g_t$. The vorticity field $\omega(x,t)$ has the properties (a) $\nabla\cdot\omega=0$, (b) $\int_M \omega dx=0$, and (c) for any $\omega$ satisfying (a) and (b) there exists unique velocity field $u$ such that $\curl u=\omega$, $\dive u=0$, and $\int_M udx=0$. In addition, if $\omega\in H^{s-1}$, then $u\in H^s$, and the operator $\curl$ establishes a continuous isomorphism between the spaces $H^s_0$ of velocities and $H^{s-1}_0$ of vorticities satisfying both (a) and (b). Let us denote the operator inverse to $\curl$ by $\curl^{-1}$.

Then the vorticity equation is obtained from (2.4) by applying the operator $\curl^{-1}$:

\be
\frac{\pa g_t(x)}{\pa t}=\left(\curl^{-1} g_{t*}\curl u(x,0)\right)\circ g_t(x).
\ee
This equation is equivalent to the Lagrange equation (2.4) and the initial conditions (2.2).

To move forward, we have to introduce  the analytical coordinates (or parameters) in the neighborhood of unity in $\D$, so that we could talk about the analyticity of the exponential map. Here we use the local Euclidean structure of the torus which saves us some additional work compared to the general case.

For any $v \in H^s_0$ consider the map $f: \ x\mapsto x+v(x)$. This map does not, in general, preserve the volume, and we correct it by adding a gradient term: $g_v(x)=x+v(x)+\grad\phi(x)$. 

\begin{lemma}
For any $v\in H^s_0$ such that $||v||_s<\eps$ there exists unique (up to an additive constant) function $\phi\in H^{s+1}$ such that the map $g_v\in\D$, i.e. its Jacobian $\det(g_v')=1$. The correspondence $v(x) \mapsto g_v(x)-x$ is analytic as a map from $H^s_0$ to $H^s(\T^3, \R^3)$.
\end{lemma}

\begin{proof}
Let $x=(x_1, x_2, x_3)$ be the locally Cartesian coordinates on the torus, $0\le x_i<2\pi$. Then the condition $g_v\in\D$ is equivalent to

\be
\det
\begin{pmatrix}
1+v_{1,1}+\phi_{,11}&v+{1,2}+\phi_{,12}&v_{1,3}+\phi_{,13}\\
v_{2,1}+\phi_{,21}&1+v_{2,2}+\phi_{,22}&v_{2,3}+\phi_{,23}\\
v_{3,1}+\phi_{,31}&v_{3,2}+\phi_{,32}&1+v_{3,3}+\phi_{,33}
\end{pmatrix}=1.
\ee
Here $v_{1,1}=\frac{\pa v_1}{\pa x_1}, \ \ \phi_{,11}=\frac{\pa^2\phi}{\pa x_1^2}$, etc. Equation (2.6) is a second order equation with respect to $\phi$. It can be written in the form

\be
\Delta\phi=P(\nabla v, \nabla^2\phi),
\ee
where $P$ is a polynomial of degree 3 whose all terms have degree 2 or 3. This is equivalent to the equation

\be
\phi=\Delta^{-1} P(\nabla v, \nabla^2\phi)=Q(v,\phi).
\ee
Operator $Q$ is analytic from $H^s_0\times H^{s+1}$ and satisfies the inequalities

\be
||Q(u,\phi)||_{s+1}\le  C(||v||_s^2+||\phi||_{s+1}^2);
\ee
\be
||\pa_v Q(v,\phi)||_{s+1}+||\pa_\phi Q(v,\phi)||_{s+1}\le C(||v||_s+||\phi||_{s+1}).
\ee
It follows then that there exists $r>0$ such that for any $v$ such that $||v||_s<r$ there exists unique solution $\phi=\phi(u)$ of equation (2.8) (this follows from the contractive map argument). Further, the analytic implicit function theorem in Banach spaces (proved below) implies that $\phi(v)$ is analytic, because $Q$ is.

Now, the map $v\in H^s_0\mapsto g_v\in\D$ is the required map defining an analytic chart on $\D$ in a neighborhood of $Id$ where

\be
g_v(x)=x+v(x)+\nabla\phi(v)(x).
\ee
\end{proof}

Now we prove that in this chart, the vorticity equation (2.5) has an analytic right hand side. The motion equations are

\beqn
&&\frac{\pa g_t}{\pa t}=u\circ g_t;\\
&&\left(\curl(u\circ g_t^{-1})\right)\circ g_t=g_t'\cdot\omega_0,
\eeqn
where $\omega_0=\curl u_0$. Let us introduce notations $g'(x)\cdot\omega_0(x)=\F_g\omega_0(x)$, and $(\curl(u\circ g^{-1}))\circ g=\G_g u(x)$. Then the velocity $u(x)$ satisfies the equation

\be
\G_g u=\F_g\omega_0.
\ee
Let us denote $u\circ g(x)$ by $U(x)$; then $u(y)=U\circ g^{-1}(y)$. Now observe that $\left(\curl(U\circ g^{-1})\right)\circ g(x)$ is for every $x$ a rational function of $g'(x)$ and $\nabla U(x)$ (linear with respect to the last argument).

For $g \in\D$, consider the space $X_g\subset H^s(M,\R^n)$ of vector functions $U(x)$ such that the field $u(y)=U(g^{-1}(y))$ is divergence free and has zero mean. This is expressed by the equation $R_g(\nabla U)=0$ where $R_g$ is a linear first order differential operator whose coefficients are rational functions of $g'$. The operator $R_g$ depends on $g \in\D$ analytically; hence the subspace $X_g$ also depends analytically on $g$.

Define another operator $\G_g\omega(x)=g'(x)\cdot\omega(x)$ which is obviously an operator from $H^{s-1}$ to $H^{s-1}$ depending analytically on $g\in\D$.

Our problem is thus equivalent to the following equations:

\beqn
&&\frac{\pa g_t}{\pa t}=U;\\
&&U\in X_g;\\
&&\F_g U=\G_g\omega.
\eeqn
Here, recall, $X_g$ is a subspace of $H^s(M,\R^3)$ depending analytically on $g$; $\G_g$ is an invertible operator from $X_g$ to $H^s_0(M,\R^3)$ depending analytically on $g$; $\F_g$ is an operator from $H^{s-1}_0(M,\R^3)$ to itself depending analytically on $g$. Hence, the flow $g_t$ satisfies the equation

\be
\frac{\pa g_t}{\pa t}=\F_{g_t}^{-1}\G_{g_t}\omega.
\ee
We solve this equation with the initial condition

\be
g_0=Id.
\ee
Now we use the following general theorem.

\begin{theorem}
Let $E, \ \ F$ be complex Banach spaces. Consider the equation

\be
\frac{d x}{d t}=f(x,y)
\ee
with the initial condition

\be
x(0)=x_0.
\ee
Here $x\in E, \ y\in F$, and $f: E\times F \to E$ is an analytic map defined in a neighborhood of $(x_0, y_0)\in E\times F$. Then there exist $r>0$ and $T>0$ such that (i) if $||y-y_0||_F<r$, then there exists unique solution $x(t,y)$ of (2.20), (2.21) for $|t|<T$; (ii) the map $(y,t)\mapsto x(y,t)$ from $B_r(y_0)\times(-T,T)$ to $E$ is analytic.
\end{theorem}

\begin{proof}
Following Lichtenstein \cite{L1}, consider the {\it complex} values of $t$: instead the segment $[-T,T]$ consider the disk $|t|\le T$. The problem (2.20), (2.21) is then equivalent to the integral equation

\be
x(y,t)=x_0+\intl_0^t f(x(y,s),y) ds.
\ee
Here $y\in B_r(y_0)$ is a parameter. The upper limit $t$ is a point of the disk $D_T: \ |t|\le T$, and the integral is along any path in $D_t$ connecting zero and $t$. We solve this equation by the Picard successive approximations:

\beqn
&&x_0(y,t)\equiv x_0;\\
&&x_1(y,t)=x_0+\intl_0^t f(x_0(y,s),y) ds;\\
&&\ldots\nonumber\\
&&x_{n+1}(y,t)=x_0+\intl_0^t f(x_n(y,s),y) ds;\\
&&\ldots\nonumber
\eeqn
Note that each iteration produces (for a fixed $y$) an analytic function $x_n(y,t)$ in $D_T$. By the usual argument of contracting maps, the sequence $x_n(y,t)$ converges to some limit $x(y,t)$ uniformly in $D_T$ as $n\to\infty$, and the limit is an analytic function in $D_T$ (because of the Cauchy formula). Thus, the solution is analytic in $t$ for every fixed $y$ In addition, the solution is bounded, $||x(y,t)||_E<C_x$ for all $y\in B_r$. It remains to prove the analytic dependence on $y\in B_r$.

We can "continue" our equation including new variables, namely $u=\frac{\pa x}{\pa y}$ and $v=\frac{\pa u}{\pa y}$. The system looks then as follows:

\beqn
&&\frac{d x}{d t}=f(x,y);\\
&&\frac{d u}{d t}=g(x,u,y);\\
&&\frac{d v}{d t}=h(x,u,v,y),
\eeqn
where $f, g, h$ are analytic maps in respective spaces. The initial conditions are $x(0)=x_0, \ u(0)=0, \ v(0)=0$. Using the previous result, we conclude that there exists unique solution $x(t,y), \ u(t,y), \ v(t,y)$ which depends analytically on $t\in D_T$, and is uniformly bounded for $y\in B_r$. This means that $x(t,y)$ is continuously differentiable with respect to $y$. It is also continuously differentiable with respect to $t$ in force of the equations. So, it is continuously differentiable with respect to $(t,y)$.

Note that the continuous differentiability holds in the complex sense. Hence, the function $x(t,y)$ is analytic in $D_T\times B_r$.
\end{proof}

For the completeness, let us prove the Analytic Implicit Function Theorem.

\begin{theorem}
Suppose $E, F, G$ are complex Banach spaces, and $\Phi: E\times F \to G$ is an analytic map defined in a neighborhood of a point $x_0, y_0)\in E\times F$, and $\Phi(x_0, y_0)=z_0$. Suppose the linear map $\frac{\pa\Phi}{\pa x}(x_0, y_0)$ is invertible, and its image is the whole space $G$. Then

(i) there exists $r>0$ such that for any $y\in F$ and any $z\in G$ such that $||y-y_0||_F\le r$ and $||z-z_0||_G\le r$ there exists unique solution $x(y,z)$ of the equation $\Phi(x,y)=z$ lying in a neighborhood of $x_0$;

(ii) The function $x(y,z)$ is analytic with respect to $(y,x)$.
\end{theorem}

\begin{proof}
(i) This is exactly the Implicit Function Theorem for the Banach spaces proved by the contraction argument.

(ii) A part of the classical Implicit Function Theorem is the regularity of solution: $x(y,z)\in C^1$. For the complex spaces $E, F, G$ the solution $x(y,z)$ is continuously differentiable in the {\it complex} sense; this implies its analyticity, like in the classical complex analysis. (Necessary details can be found in \cite{HP}.)
\end{proof}

\bs

{\bf Remarks. } (1) Theorem 1 is inherently global; for local solution of the Euler equations (defined in a neighborhood of a point $x_0\in M$) the analyticity of trajectories does not universally hold. Consider, for example, the following solution:

\be
u(x,t)\equiv w(t); \ \ \ p(x,t)=\dot w(t)\cdot (x-x_0),
\ee 
where $w(t)$ is an arbitrary vector-function of time. 

(2) We can consider also the group exponential map $\exp: T_{Id} \D \to \D$, $u\in T_{Id}\D\mapsto \exp_u\in\D$ defined by $\exp_u(x)=\xi_1(x)$ where $\xi_t\in\D$ is a solution of the equation: $\frac{d \xi_t(x)}{dt}=u(\xi_t(x)), \ \ \xi_0(x)=x$. This map looks superficially like the geodesic exponential map $Exp$. However, as it was pointed out by Milnor \cite{Mil}, the map $\exp$ is neither analytic, nor even $C^1$. In fact, $\exp(T_{Id}\D)$ contains no neighborhood of $Id$ (unlike the geodesic exponential map $Exp$).

\end{document}